\documentclass[11pt,a4paper]{article}
\usepackage[margin=25mm]{geometry}
\usepackage[T1]{fontenc}
\usepackage{lmodern}
\usepackage{amsmath,amssymb,amsthm,mathtools}
\usepackage{enumitem,booktabs,needspace}
\usepackage[hidelinks]{hyperref}
\allowdisplaybreaks[1]
\setlist[enumerate]{itemsep=3pt,topsep=5pt}
\newtheorem{theorem}{Theorem}[section]
\newtheorem{proposition}[theorem]{Proposition}
\newtheorem{lemma}[theorem]{Lemma}
\newtheorem{corollary}[theorem]{Corollary}
\theoremstyle{definition}
\newtheorem{definition}[theorem]{Definition}

\DeclareMathOperator{\ex}{ex}
\newcommand{\tht}{\theta}
\newcommand{\CC}{\mathcal C}
\newcommand{\OO}{\mathcal O}
\newcommand{\QQ}{\mathbb Q}
\newcommand{\RR}{\mathbb R}
\newcommand{\sG}{\dot G}
\newcommand{\sF}{\dot F}
\newcommand{\sR}{\dot R}
\newcommand{\sT}{\dot T}

\title{Tur\'{a}n coefficients of signed graphs}
\author{{Junjie Wang$^{a}$, Yaoping Hou$^{a,}$\footnote{Corresponding author.}\setcounter{footnote}{-1}\footnote{\emph{E-mail address:} yphou@hunnu.edu.cn}}\\[2mm]
	\small $^{a}$College of Mathematics and Statistics, Hunan Normal University,\\
	\small Changsha,  Hunan, 410081, China\\[2mm]
}
\date{}
\begin{document}
\maketitle

\begin{abstract}
We use spectral methods to determine Tur\'an coefficients of signed graphs under containment that preserves negative edges and allows positive edges to be realized by either sign. For a fixed forbidden signed graph $\sF$ with at least one edge, the coefficient $\tht(\sF)$ is the supremum of the weights $\rho\ge1$ for which every $\sF$-free signed graph $\sG$ of order $n$ satisfies $e^+(\sG)+\rho e^-(\sG)\le\binom n2+o(n^2)$. We prove a variational characterization of the weighted extremal density and show that $\tht(\sF)=1+1/L(\sF)$, where $L(\sF)$ is the supremum of the adjacency spectral radii of $\sF$-admissible graphs. The coefficient is infinite exactly when the underlying graph $F$ is bipartite, and equals one exactly when $F$ has no independent set meeting every negative edge. When $\chi(F)=3$, the possible coefficients are exactly $1$, $3/2$, $1+1/\sqrt3$, $(1+\sqrt5)/2$, $1+1/\sqrt2$, and $2$, each of which is realized by a signed graph. In this case, if $\tht(\sF)>1$, then $L(\sF)$ is attained by an admissible graph on at most five vertices, and the order bound is sharp. More generally, when some path is inadmissible, we obtain explicit order bounds for connected admissible graphs with bounded maximum degree. We also derive restrictions on the algebraic degree and rational values of $\tht(\sF)$.
\end{abstract}


\section{Introduction}\label{sec:introduction}

A signed graph $\sG=(G,\sigma_G)$ consists of an ordinary graph $G$ and a signature $\sigma_G:E(G)\to\{+1,-1\}$. Write $e^+(\sG)$ and $e^-(\sG)$ for its numbers of positive and negative edges.

For an ordinary graph $F$ with $\chi(F)=r\ge3$, the Erd\H{o}s--Stone theorem~\cite{ErdosStone} gives
\[
 \ex(n,F)=\left(1-\frac1{r-1}+o(1)\right)\binom n2.
\]
Here the forbidden signed graph $\sF$ has at least one edge. We study containment in which every negative edge of $\sF$ must be realized by a negative edge, whereas a positive edge may be realized by either sign. Equivalently, copies may be obtained by deleting vertices and edges and changing negative edges to positive. If $\sF$ has a negative edge, the all-positive complete graph is $\sF$-free, so total edge density alone gives no nontrivial bound. We instead consider the supremum $\tht(\sF)$ of the fixed weights $\rho\ge1$ for which
\[
 e^+(\sG)+\rho e^-(\sG)\le\binom n2+o(n^2)
\]
holds uniformly over all $n$-vertex $\sF$-free signed graphs. Since the left-hand side is $e(G)+(\rho-1)e^-(\sG)$, this inequality relates negative edges to missing edges.

A simple interpretation uses two thresholds for connection strength, motivated by the distinction between strong and weak ties~\cite{Granovetter,Onnela}. Given symmetric strengths $w_{uv}\ge0$ and $0<a<b$, label an edge positive when $a\le w_{uv}<b$ and negative when $w_{uv}\ge b$. A copy then requires strength at least $a$ on every edge and at least $b$ on its negative edges. Equivalently, regard $E(G)$ and $M=E^-(\sG)$ as two layers, with $M\subseteq E(G)$~\cite{Kivela}. The objective is $|E(G)|+(\rho-1)|M|$. 

\subsection{Related work}
Balance and switching are central to signed graph theory~\cite{Harary,Zaslavsky}. Switching at a vertex set $X$ reverses the signs of edges with exactly one endpoint in $X$ and preserves the sign of every cycle. Extremal problems for unbalanced signed graphs include spectral radius and index~\cite{BrunettiStanic}, forbidden negative triangles~\cite{WangHouLi}, negative odd cycles~\cite{WangHouHuang}, and unbalanced complete subgraphs~\cite{ChenYuan,WangK5,XiongHou}. We consider containment with fixed signatures, in which negative edges are preserved and positive edges may be represented by edges of either sign, with no restriction on the balance of the signed graph; see Definition~\ref{def:containment} for the precise definition.

Mani and Yu's Tur\'an coefficients for mixed graphs~\cite{MY} provide a related weighted framework. Orienting the negative edges of $\sF$ in all possible ways and leaving its positive edges undirected gives a finite family $\OO(\sF)$. For any fixed orientation of the negative edges of $\sG$, the resulting mixed graph avoids $\OO(\sF)$ if and only if $\sG$ is $\sF$-free. When $\sF$ has a negative edge and $1<\tht(\sF)<\infty$, their finite-family variational theory~\cite[Theorem~4.13 and Section~5.3]{MY}, together with the reduction in Section~\ref{sec:componentbounds}, gives finite attainment and algebraicity. The all-positive case follows directly from the chromatic number. We give an independent proof of the spectral characterization and use it to obtain structural classifications and explicit bounds.

\subsection{Main results}\label{subsec:mainresults}
For an ordinary graph $S$, let $\sT(S)$ be the complete signed graph whose negative edges are precisely $E(S)$. A blow-up replaces vertices by independent sets and retains the prescribed signs between these sets. Let $\mathcal A(\sF)$ be the class of ordinary graphs $S$ with at least one vertex for which every blow-up of $\sT(S)$ is $\sF$-free. Such graphs are called $\sF$-admissible. Write
\[
 L(\sF)=\sup_{S\in\mathcal A(\sF)}\lambda_1(S),
\]
where $\lambda_1(S)$ is the adjacency spectral radius of $S$. The limiting maximum of $(e^+(\sG)+\rho e^-(\sG))/\binom n2$ over $n$-vertex $\sF$-free signed graphs is denoted by $h_{\sF}(\rho)$.

For a graph $S$ of order $s$, define
\[
 \Delta_{s-1}=\left\{x\in\RR^s:x_i\ge0,\ \sum_{i=1}^{s}x_i=1\right\},
 \qquad M_\rho(S)=J-I+(\rho-1)A(S),
\]
\[
 g_\rho(S)=\max_{x\in\Delta_{s-1}}x^{\mathsf T}M_\rho(S)x,
\]
where $A(S)$ is the adjacency matrix, $J$ is the all-ones matrix, and $I$ is the identity matrix of order $s$. Our first result characterizes $\tht(\sF)$ through the adjacency spectral radii of admissible graphs.

\begin{theorem}\label{thm:spectral}
For every signed graph $\sF$ with at least one edge and every fixed $\rho\ge1$,
\[
 h_{\sF}(\rho)=\sup_{S\in\mathcal A(\sF)}g_\rho(S).
\]
Consequently,
\[
 \tht(\sF)=1+\frac1{L(\sF)},
\]
with the conventions $1/0=\infty$ and $1/\infty=0$.
\end{theorem}

Let $\CC(\sF)$ be the family of independent sets in $F$ meeting every negative edge of $\sF$. When this family is nonempty, define
\[
 k(\sF)=\min_{C\in\CC(\sF)}\chi(F-C).
\]
The next result identifies the extreme cases and bounds all other coefficients.

\begin{theorem}\label{thm:structure}
For every signed graph $\sF$ with at least one edge:
\begin{enumerate}[label=\textup{(\roman*)}]
\item $\tht(\sF)=\infty$ if and only if $F$ is bipartite.
\item $\tht(\sF)=1$ if and only if $\CC(\sF)=\varnothing$.
\item If $r=\chi(F)\ge3$ and $\CC(\sF)\ne\varnothing$, then $r-1 \le k(\sF)\le r$, $r-2\le L(\sF)\le k(\sF)-1$, and hence
\[
 1+\frac1{k(\sF)-1}\le\tht(\sF)\le1+\frac1{r-2}.
\]
\end{enumerate}
In particular, if $k(\sF)=r-1$, then $\tht(\sF)=(r-1)/(r-2)$.
\end{theorem}

Let $P_t$ and $C_t$ denote the path and cycle of order $t$, respectively. For three-chromatic underlying graphs, this spectral approach yields a complete classification.

\begin{theorem}\label{thm:three}
Let $\chi(F)=3$.
\begin{enumerate}[label=\textup{(\roman*)}]
\item If $\CC(\sF)=\varnothing$, then $\tht(\sF)=1$.
\item If $\CC(\sF)\ne\varnothing$ and at least one of $C_4,C_5$ is $\sF$-admissible, then $L(\sF)=2$ and $\tht(\sF)=3/2$.
\item If $\CC(\sF)\ne\varnothing$ and neither $C_4$ nor $C_5$ is admissible, let $t=\min\{s:P_s\notin\mathcal A(\sF)\}$. Then $t$ exists, $3\le t\le6$, and
\[
 \begin{aligned}
 L(\sF)&=2\cos\frac\pi t,\\
 \tht(\sF)&=1+\frac1{2\cos(\pi/t)}.
 \end{aligned}
\]
\end{enumerate}
The possible coefficients are exactly
\[
 \left\{1,\frac32,1+\frac1{\sqrt3},\frac{1+\sqrt5}{2},1+\frac1{\sqrt2},2\right\},
\]
and every value occurs.
If $\tht(\sF)>1$, then $L(\sF)$ is attained by an admissible graph of order at most five, and the order bound is sharp. The values $t=3,4,5,6$ all occur in (iii), and no $t\ge7$ occurs there.
\end{theorem}

The proofs appear in Sections~\ref{sec:spectral}--\ref{sec:three}. Section~\ref{sec:componentbounds} gives explicit bounds on the order of connected admissible graphs when $\CC(\sF)\ne\varnothing$ and some path is inadmissible, and derives consequences for algebraic degrees and rational coefficients.

\section{Definitions and preliminaries}\label{sec:preliminaries}

\subsection{Signed graphs and containment}
All graphs are finite, simple, and undirected. A signed graph is denoted by $\sG=(G,\sigma_G)$, where $G=(V(G),E(G))$ is its underlying graph and $\sigma_G:E(G)\to\{+1,-1\}$ is its signature. Write $E^{\pm}(\sG)=\{e\in E(G):\sigma_G(e)=\pm1\}$, $e^{\pm}(\sG)=|E^{\pm}(\sG)|$. Thus $e(G)=e^+(\sG)+e^-(\sG)$. Vertex sets, degrees, independent sets, and proper colorings refer to the underlying graphs; in particular, $\chi(F)$ is the ordinary chromatic number of $F$. For $X\subseteq V(G)$, the induced signed subgraph $\sG[X]$ has underlying graph $G[X]$ and the restriction of $\sigma_G$ to its edges.

\begin{definition}\label{def:containment}
Let $\sF=(F,\sigma_F)$ and $\sG=(G,\sigma_G)$ be signed graphs. We write $\sF\subseteq_w\sG$ if there is an injective map $f:V(F)\to V(G)$ such that
\begin{align*}
 uv\in E^-(\sF)&\ \Longrightarrow\ f(u)f(v)\in E^-(\sG),\\
 uv\in E^+(\sF)&\ \Longrightarrow\ f(u)f(v)\in E(G).
\end{align*}
The chosen vertices and corresponding edges form a copy of $\sF$ in $\sG$. If there is no such copy, then $\sG$ is $\sF$-free.
\end{definition}

The two signs thus specify requirements of different strength. For example, if $(K_3,+)$ and $(K_3,-)$ denote the all-positive and all-negative triangles, respectively, then $(K_3,+)\subseteq_w(K_3,-)$, $(K_3,-)\not\subseteq_w(K_3,+)$. Equivalently, $\sF$ can be obtained from $\sG$ by deleting vertices and edges and changing some negative edges to positive. This operation acts on the given signature and differs from switching. Throughout, the fixed forbidden signed graph $\sF$ has at least one edge; isolated vertices are allowed.

\subsection{Weighted extremal functions and Tur\'an coefficients}
For $n\ge2$, let $N_n=\binom n2$. For a signed graph $\sG$ of order $n$, define
\[
 \alpha(\sG)=\frac{e^+(\sG)}{N_n},
 \qquad \beta(\sG)=\frac{e^-(\sG)}{N_n}.
\]
For a fixed real number $\rho\ge1$, define
\[
 \ex_\rho(n,\sF)
 =\max\{e^+(\sG)+\rho e^-(\sG):|V(G)|=n,\ \sG\text{ is }\sF\text{-free}\}.
\]
Choose an extremal graph and then a uniformly random set of $m$ vertices, where $n\ge m\ge2$. Its induced signed subgraph is still $\sF$-free, so $\ex_\rho(n,\sF) / \binom n2
 \le \ex_\rho(m,\sF) / \binom m2$. The normalized extremal numbers are nonincreasing and nonnegative. Hence the limit $h_{\sF}(\rho)=\lim_{n\to\infty}\frac{\ex_\rho(n,\sF)}{N_n}$ exists.

\begin{definition}\label{def:coefficient}
The Tur\'an coefficient of $\sF$ is
\[
 \tht(\sF)=\sup\{\rho\ge1:h_{\sF}(\rho)\le1\},
\]
with the value $\infty$ allowed.
\end{definition}

Since $e^-(\sG)\le N_n$, the function $h_{\sF}$ is nondecreasing and satisfies $|h_{\sF}(\rho)-h_{\sF}(\rho')|\le|\rho-\rho'|$. Also $h_{\sF}(1)\le1$. Thus the admissible weights form an interval containing $1$, and this interval contains its right endpoint when $\tht(\sF)$ is finite. The condition $h_{\sF}(\rho)\le1$ means that for every $\varepsilon>0$ there is an $n_0=n_0(\sF,\rho,\varepsilon)$ such that every $\sF$-free signed graph $\sG$ of order $n\ge n_0$ satisfies
\[
 e^+(\sG)+\rho e^-(\sG)\le(1+\varepsilon)N_n.
\]
Here $\rho$ is fixed before $n$ tends to infinity. Moreover,
\[
 \tht(\sF)=\infty
 \iff
 \max_{\substack{|V(G)|=n\\\sG\text{ is }\sF\text{-free}}}\beta(\sG)=o(1).
\]

\subsection{Admissible graphs and spectral notation}
We write $\sF\to_w\sR$ if a map $f:V(F)\to V(R)$ satisfies the two edge conditions in Definition~\ref{def:containment}, without requiring injectivity. Since $R$ is loopless, adjacent vertices of $F$ must be assigned distinct vertices of $R$.

For an ordinary graph $S$, recall that the complete signed graph $\sT(S)$ is defined by
\[
 \sT(S)=(K_{|V(S)|},\sigma_S),
 \qquad
 \sigma_S(ij)=
 \begin{cases}
 -1,&ij\in E(S),\\
 +1,&ij\notin E(S),
 \end{cases}
\]
where the complete graph has vertex set $V(S)$. In particular, the negative edges of $\sT(P_t)$ form a spanning path, and all remaining edges are positive.

A blow-up of $\sR$ replaces each vertex $i$ by an independent set $V_i$ and each edge $ij$ by all edges of sign $\sigma_R(ij)$ between $V_i$ and $V_j$; a nonedge produces no edges between the corresponding parts. Then $\sF\to_w\sR$ if and only if $\sF$ occurs in some blow-up of $\sR$. A copy gives the map assigning each vertex to its part. Conversely, sufficiently large parts allow distinct representatives for all vertices assigned to each part.

\begin{definition}\label{def:admissible}
An ordinary graph $S$ with at least one vertex is $\sF$-admissible if $\sF\not\to_w\sT(S)$. Denote the class of such graphs by $\mathcal A(\sF)$. Equivalently, every blow-up of $\sT(S)$ is $\sF$-free.
\end{definition}

For an ordinary graph $S$, let $A(S)$ be its adjacency matrix and $\lambda_1(S)$ its largest eigenvalue. By the Perron--Frobenius theorem, the largest eigenvalue of a nonnegative real symmetric matrix equals its spectral radius and has a nonnegative eigenvector. Thus $\lambda_1(S)$ is also the adjacency spectral radius of $S$. The matrix $A(S)$ used here has entries in $\{0,1\}$; it differs from the signed adjacency matrix $A(\sG)$, whose entries lie in $\{0,+1,-1\}$.

Define
\[
 L(\sF)=\sup\{\lambda_1(S):S\in\mathcal A(\sF)\}\in[0,\infty].
\]
A finite value $L(\sF)$ is attained if $\lambda_1(S_*)=L(\sF)$ for some $S_*\in\mathcal A(\sF)$.

Admissibility is inherited by ordinary subgraphs with at least one vertex. If $S'\subseteq S$ and $|V(S')|\ge1$, inclusion gives $\sT(S')\to_w\sT(S)$: negative edges remain negative, while positive edges may become negative. Hence $S\in\mathcal A(\sF)$ implies $S'\in\mathcal A(\sF)$.

\subsection{Regularity}
For a bipartite graph $H$ with parts $U,V$, define
\[
 d_H(X,Y)=\frac{e_H(X,Y)}{|X||Y|}
 \qquad(\varnothing\ne X\subseteq U,\ \varnothing\ne Y\subseteq V).
\]
We call $H$ $\gamma$-regular if
\[
 |d_H(X,Y)-d_H(U,V)|<\gamma
\]
whenever $|X|\ge\gamma|U|$ and $|Y|\ge\gamma|V|$.

For an edge-colored complete graph, regularity in a color refers to the bipartite graph consisting of the edges of that color. We use the following form of the many-color regularity lemma of Koml\'os and Simonovits~\cite[Theorem~1.18]{KS}.

\begin{lemma}\label{lem:multicolor-regularity}
For every $0<\varepsilon<1$ and positive integers $q,k_0$, there are integers $K,n_0$ such that every coloring of the edges of $K_n$ with $q$ colors, where $n\ge n_0$, admits a partition
\[
 V(K_n)=V_0\sqcup V_1\sqcup\cdots\sqcup V_k
\]
with $|V_0|\le\varepsilon n$, $k_0\le k\le K$, and $|V_1|=\cdots=|V_k|>0$, such that all but at most $\varepsilon k^2$ choices $1\le i<j\le k$ give an $\varepsilon$-regular bipartite graph in every color.
\end{lemma}

The cited version tests subsets of size strictly greater than $\varepsilon$ times the size of their part. Applying it with parameter $\varepsilon/2$ gives the convention used here.

\section{Spectral characterization}\label{sec:spectral}

We prove Theorem~\ref{thm:spectral} using the following two lemmas.
The following lemma uses the Motzkin--Straus coordinate-merging argument~\cite{MS} to reduce the quadratic maximization to a clique of the reduced graph.

\begin{lemma}\label{lem:compression}
Let $\rho\ge1$ and let $\sR$ be a signed graph of order $k\ge1$. Define $W_\rho(\sR)$ to have diagonal entries zero and off-diagonal entries $0$, $1$, or $\rho$ according as the pair is missing, positive, or negative. The maximum of $x^{\mathsf T}W_\rho(\sR)x$ over $x\in\Delta_{k-1}$ has a maximizer whose support is a clique in $R$.
\end{lemma}

\begin{proof}
Write $W=W_\rho(\sR)$ and choose a maximizer $x$ with as few positive coordinates as possible. Such a vector exists because the simplex is compact. If its support is not a clique, choose nonadjacent vertices $i,j$ with $x_i,x_j>0$. Let $c=x_i+x_j$ and $Q=V(R)\setminus\{i,j\}$. Replacing $(x_i,x_j)$ by $(u,c-u)$, where $0\le u\le c$, changes the quadratic form to
\[
 \sum_{p\in Q}\sum_{q\in Q}W_{pq}x_p x_q
 +2u\sum_{q\in Q}W_{iq}x_q
 +2(c-u)\sum_{q\in Q}W_{jq}x_q.
\]
This expression is affine in $u$, so its value at one of the endpoints $u=0,c$ is at least its value at $u=x_i$. That endpoint therefore gives another maximizer with fewer positive coordinates, a contradiction.
\end{proof}

The next lemma adapts the greedy proof of the Key Lemma of Koml\'os and Simonovits~\cite[Theorem~2.1]{KS}. We apply regularity either to the negative-edge bipartite graph or to the bipartite graph containing both signs. The proof also allows several vertices of $\sF$ to be assigned to the same part.

\begin{lemma}\label{lem:embedding}
Let $m=|V(F)|$, $0<\delta\le1$, and $a=\delta/2$. Suppose
\[
 0<\gamma\le\min\left\{\frac\delta2,\frac{a^m}{2(m+1)}\right\},
 \qquad t>\frac{2m}{a^m}.
\]
Let $\sR$ be a signed graph on $[k]$, and let $V_1,\ldots,V_k$ be disjoint sets of size $t$ in a signed graph $\sG$. For every negative edge $ij$ of $\sR$, assume that the bipartite graph consisting of the negative edges between $V_i,V_j$ is $\gamma$-regular and has density at least $\delta$. For every positive edge $ij$ of $\sR$, assume the same properties for the bipartite graph consisting of all edges between $V_i,V_j$.

Suppose each vertex $v$ of $\sF$ is assigned an index $c(v)\in[k]$, so that the endpoints of every edge of $\sF$ are assigned adjacent vertices of $\sR$, and the endpoints of every negative edge are assigned a negative edge of $\sR$. Then $\sG$ contains a copy of $\sF$ in which the representative of $v$ belongs to $V_{c(v)}$ and all representatives are distinct.
\end{lemma}

\begin{proof}
For each edge $ij$ of $\sR$, let $B_{ij}$ be the bipartite graph specified in the hypothesis. Thus $B_{ij}$ uses only negative edges when $ij$ is negative, and both signs when $ij$ is positive. We choose the representatives one at a time. For each vertex $v$ not yet chosen, let $C_v\subseteq V_{c(v)}$ consist of the vertices adjacent in the appropriate $B_{c(u)c(v)}$ to the chosen representative of every already chosen neighbor $u$ of $v$. Initially $C_v=V_{c(v)}$. Used vertices are excluded when the next representative is selected, rather than removed from these sets.

After $r_v$ neighbors of $v$ have been chosen, we ensure inductively that
\[
 |C_v|\ge(\delta-\gamma)^{r_v}t\ge a^m t\ge\gamma t.
\]
This holds initially. Consider the next vertex $u$. For each unchosen neighbor $v$ of $u$, fewer than $\gamma t$ vertices in $V_{c(u)}$ have fewer than $(\delta-\gamma)|C_v|$ neighbors in $C_v$ in $B_{c(u)c(v)}$. Otherwise that set of vertices and $C_v$ would violate $\gamma$-regularity, since the density of $B_{c(u)c(v)}$ is at least $\delta$.

There are at most $m$ such neighbors and at most $m$ used vertices. The number of choices in $C_u$ avoiding all these exclusions is at least
\[
 |C_u|-m\gamma t-m
 \ge \frac{a^m t}{2}-m>0.
\]
Choose one of them, say $z$. For each unchosen neighbor $v$, replace $C_v$ by $C_v\cap N_{B_{c(u)c(v)}}(z)$. The new size is at least $(\delta-\gamma)|C_v|$, which proves the induction step. The inequalities $\delta-\gamma\ge a$, $r_v\le m$, and $a\le1$ justify the displayed lower bounds. The construction selects distinct representatives even when several vertices are assigned to one part. Every negative edge of $\sF$ is realized in a negative $B_{ij}$, and every positive edge is realized by an edge of either sign, as required.
\end{proof}

\begin{proof}[Proof of Theorem~\ref{thm:spectral}]
For the lower bound, fix $S\in\mathcal A(\sF)$ of order $s$ and $x\in\Delta_{s-1}$. Take a blow-up $\sG_n$ of $\sT(S)$ whose part sizes are $x_i n+O(1)$. It is $\sF$-free by admissibility, and hence
\[
 h_{\sF}(\rho)\ge\lim_{n\to\infty}
 \frac{e^+(\sG_n)+\rho e^-(\sG_n)}{\binom n2}
 =x^{\mathsf T}M_\rho(S)x.
\]
Taking the maximum over $x$ and the supremum over $S$ gives the lower bound.

For the reverse inequality, fix $\eta>0$ and let $m=|V(F)|$. Choose $0<\delta\le1$, then $k_0$, and then $\varepsilon>0$, so that
\[
 \rho(\delta+4\varepsilon+1/k_0)<\eta,
\]
\[
 2\varepsilon\le\min\left\{\frac\delta2,
 \frac{(\delta/2)^m}{2(m+1)}\right\}.
\]
Let $\sG$ be an $n$-vertex $\sF$-free signed graph. Color each edge of the complete graph on $V(G)$ according to whether it is a positive edge, a negative edge, or a nonedge of $\sG$. Apply Lemma~\ref{lem:multicolor-regularity} with $q=3$ and parameter $\varepsilon$ to obtain a partition
\[
 V(G)=V_0\sqcup V_1\sqcup\cdots\sqcup V_k,
\]
where $|V_0|\le\varepsilon n$, $k_0\le k\le K$, and $|V_i|=t$ for $1\le i\le k$, such that at most $\varepsilon k^2$ choices $i<j$ fail to be $\varepsilon$-regular in all three colors. Here $K$ is independent of $n$. We take $n$ large enough that $t>2m/(\delta/2)^m$.

For each of the other choices $i<j$, write
\[
 d^\pm_{ij}=\frac{e^\pm_{\sG}(V_i,V_j)}{t^2},
 \qquad d_{ij}=d^+_{ij}+d^-_{ij}.
\]
The bipartite graph consisting of all edges between \(V_i\) and \(V_j\) is \(2\varepsilon\)-regular. Indeed, for any \(X\subseteq V_i\) and \(Y\subseteq V_j\) with \(|X|,|Y|\ge 2\varepsilon t\), the absolute difference between \(d(X,Y)\) and \(d(V_i,V_j)\) is at most the sum of the corresponding density differences for positive and negative edges, and is therefore less than \(2\varepsilon\). The negative-edge bipartite graph is also $2\varepsilon$-regular.

Define a signed graph $\sR$ on $[k]$ as follows. Whenever the bipartite graphs between $V_i$ and $V_j$ are regular in all three colors, include a negative edge $ij$ if $d^-_{ij}\ge\delta$, and a positive edge $ij$ if $d^-_{ij}<\delta$ and $d_{ij}\ge\delta$. Leave $i,j$ nonadjacent in every other case. If $\sF\to_w \sR$, Lemma~\ref{lem:embedding} with $\gamma=2\varepsilon$ gives a copy of $\sF$ in $\sG$. Thus, if \(\dot G\) is \(\dot F\)-free, then \(\dot F\not\to_w\dot R\).

We next bound the weighted edge count. On a negative edge of $\sR$, the weighted density between the corresponding parts is at most $\rho$. On a positive edge it is at most $1+(\rho-1)\delta$. On a regular pair assigned no edge it is at most $\rho\delta$. The excess over the corresponding entry of $W_\rho(\sR)$ is therefore at most $\rho\delta$ on each regular pair. Summing gives an excess at most $\rho\delta n^2/2$. Edges meeting $V_0$, edges on the irregular pairs, and edges inside the $V_i$ contribute at most
\[
 \rho\varepsilon n^2,\qquad
 \rho\varepsilon k^2t^2,\qquad
 \rho kt^2/2,
\]
respectively. For $x=(1/k,\ldots,1/k)$, it follows that
\[
 \frac{e^+(\sG)+\rho e^-(\sG)}{\binom n2}
 \le\frac{n}{n-1}\left[x^{\mathsf T}W_\rho(\sR)x
       +\rho(\delta+4\varepsilon+1/k_0)\right].
\]
By Lemma~\ref{lem:compression}, the quadratic form is bounded by its maximum on a clique support. The signed graph on that support is $\sT(S)$ for some ordinary graph $S$, and $S$ is admissible since $\sF\not\to_w \sR$. Therefore
\[
 h_{\sF}(\rho)\le\sup_{S\in\mathcal A(\sF)}g_\rho(S)+\eta.
\]
The first assertion follows by letting \(\eta\to0\).

For the coefficient, the Rayleigh principle gives, for every $S\in\mathcal A(\sF)$ and $x\in\Delta_{|V(S)|-1}$,
\begin{align*}
 x^{\mathsf T}M_\rho(S)x
 &=1-\|x\|_2^2+(\rho-1)x^{\mathsf T}A(S)x\\
 &\le1+\big((\rho-1)\lambda_1(S)-1\big)\|x\|_2^2.
\end{align*}
Suppose first that $0<L(\sF)<\infty$. If $1\le\rho\le1+1/L(\sF)$, then
\[
 (\rho-1)\lambda_1(S)\le(\rho-1)L(\sF)\le1,
\]
and therefore
\[
 h_{\sF}(\rho)=\sup_{S\in\mathcal A(\sF)}g_\rho(S)\le1.
\]
If $\rho>1+1/L(\sF)$, the definition of the supremum gives $S\in\mathcal A(\sF)$ with $\lambda_1(S)>1/(\rho-1)$. Let $y\ne0$ be a nonnegative eigenvector for $\lambda_1(S)$ and define $x=y/(\sum_i y_i)$. Then
\[
 h_{\sF}(\rho)\ge x^{\mathsf T}M_\rho(S)x
 =1+\big((\rho-1)\lambda_1(S)-1\big)\|x\|_2^2>1.
\]
Thus $\tht(\sF)=1+1/L(\sF)$ in this case.

If $L(\sF)=0$, every admissible $S$ has $\lambda_1(S)=0$, so the same Rayleigh bound gives $h_{\sF}(\rho)\le1$ for every fixed $\rho\ge1$, and $\tht(\sF)=\infty$. If $L(\sF)=\infty$, for every $\rho>1$ we may choose $S$ with $\lambda_1(S)>1/(\rho-1)$ and repeat the eigenvector argument to obtain $h_{\sF}(\rho)>1$. Since $h_{\sF}(1)\le1$, this gives $\tht(\sF)=1$.
\end{proof}

Section~\ref{sec:componentbounds} discusses attainment and gives an explicit bound on the order of each connected admissible graph when \(1<\theta(\dot F)<\infty\) and \(\dot F\to_w\dot T(P_t)\) for some \(t\).

\section{Structural characterization and bounds}\label{sec:structure}
Recall that an independent negative-edge cover is a set $C\in\CC(\sF)$. Since $C$ is independent in the underlying graph $F$, every negative edge has exactly one endpoint in $C$. We first prove Theorem~\ref{thm:structure} and then establish a reduction used in the three-chromatic case.

\begin{proof}[Proof of Theorem~\ref{thm:structure}]
First we prove (i). If $F$ is bipartite, a proper two-coloring gives a map $\sF\to_w\sT(K_2)$. Every ordinary graph with an edge is therefore inadmissible. Thus $L(\sF)=0$ and $\tht(\sF)=\infty$. If \(\chi(F)=r\ge3\), then \(K_{r-1}\) is \(\dot F\)-admissible, since a map \(\dot F\to_w\dot T(K_{r-1})\) would induce a proper \((r-1)\)-coloring of \(F\), contradicting \(\chi(F)=r\). Hence
\[
L(\dot F)\ge \lambda_1(K_{r-1})=r-2>0.
\]
By the spectral characterization,
\[
\theta(\dot F)=1+\frac{1}{L(\dot F)}
\le 1+\frac{1}{r-2}<\infty.
\]
Thus \(\theta(\dot F)=\infty\) only if \(F\) is bipartite, completing the proof of (i).

For (ii), we first observe that
\begin{equation}\tag{1}\label{eq:star-test}
 \sF\to_w \sT(K_{1,d})
 \quad\Longleftrightarrow\quad
 \text{there is }C\in\CC(\sF)\text{ with }\chi(F-C)\le d.
\end{equation}
For the forward implication, let \(C\) consist of the vertices of \(F\) assigned to the center of the star. Then \(C\) is independent and meets every negative edge. The assignment of the remaining vertices to the leaves gives a proper coloring of \(F-C\). Conversely, choose such a set \(C\) and a proper coloring of \(F-C\) using at most as many colors as there are leaves. Assign all vertices of \(C\) to the center and each color class to a distinct leaf. Every negative edge is then represented by a negative edge between the center and a leaf, while the endpoints of every positive edge are assigned to distinct vertices.

If $\CC(\sF)$ is empty, every star $K_{1,d}$ is admissible. Since its spectral radius is $\sqrt d$, we have $L(\sF)=\infty$, giving $\tht(\sF)=1$. If $\CC(\sF)$ is nonempty, \eqref{eq:star-test} shows that $K_{1,k(\sF)}$ is inadmissible. Therefore every admissible graph $S$ satisfies $\Delta(S) \le k(\sF)-1$, $\lambda_1(S)\le\Delta(S)\le k(\sF)-1$. It follows that $\tht(\sF)>1$, with the value $\infty$ allowed. This proves (ii).

For (iii), if $C$ is independent, $\chi(F)\le\chi(F-C)+1$, while $\chi(F-C)\le\chi(F)$. Minimizing over the covers gives $r-1\le k(\sF)\le r$. The spectral bounds now follow from the admissible clique and the degree bound, and Theorem~\ref{thm:spectral} gives (iii).
\end{proof}

There is a direct extremal construction behind part (ii). If \(\mathcal C(\dot F)=\varnothing\), construct a signed graph \(\dot G\) consisting of an all-positive clique on \(X\), an independent set \(Y\), and all negative edges between \(X\) and \(Y\). If \(\dot G\) contained a copy of \(\dot F\), the vertices of \(F\) placed in \(Y\) would form an independent set meeting every negative edge of \(\dot F\), contradicting \(\mathcal C(\dot F)=\varnothing\). Thus \(\dot G\) is \(\dot F\)-free. Writing $|X|/n\to x\in(0,1)$ gives limiting densities $\alpha=x^2$ and $\beta=2x(1-x)$, and therefore
\[
 \frac{1-\alpha}{\beta}=\frac{1+x}{2x}\longrightarrow1
 \quad\text{as }x\longrightarrow 1.
\]
For every fixed $\rho>1$, choosing $x$ sufficiently close to one makes $\alpha+\rho\beta>1$. 

The following lemma reduces the $C_4$ case to $P_6$.
\begin{lemma}\label{lem:twostars}
Suppose $\CC(\sF)\ne\varnothing$ and $\sF\to_w \sT(K_{a,b})$, where $a,b\ge1$. Then
\[
 \sF\to_w \sT(K_{1,a}\mathbin{\dot\cup}K_{1,b}).
\]
In particular, if $\sF\to_w \sT(C_4)$, then $\sF\to_w \sT(P_6)$.
\end{lemma}
\begin{proof}
Let \(X\) and \(Y\) be the two vertex classes of \(K_{a,b}\), with \(|X|=a\) and \(|Y|=b\). Let \(f\) be the given map, and choose \(C\in\mathcal C(\dot F)\).
Construct an ordinary graph \(S\) consisting of two disjoint stars, one with center \(c_X\) and leaves labelled by the vertices of \(Y\), and the other with center \(c_Y\) and leaves labelled by the vertices of \(X\). Thus $S\cong K_{1,a}\mathbin{\dot\cup}K_{1,b}$. For each \(v\notin C\), assign \(v\) to the leaf labelled \(f(v)\). For each \(v\in C\), assign \(v\) to \(c_X\) if \(f(v)\in X\), and to \(c_Y\) if \(f(v)\in Y\). We verify that this assignment satisfies the edge conditions in \(\dot T(S)\). Since \(C\) is independent and meets every negative edge, each negative edge has exactly one endpoint in \(C\). Moreover, \(f\) assigns its endpoints to opposite vertex classes of \(K_{a,b}\). The new assignment therefore places these endpoints at a center and a leaf of the same star, which are joined by a negative edge in \(\dot T(S)\).
Now consider a positive edge \(uv\). If neither endpoint belongs to \(C\), then \(f(u)\ne f(v)\), so \(u\) and \(v\) are assigned to distinct leaves. If exactly one endpoint belongs to \(C\), they are assigned to a center and a leaf, respectively. Both endpoints cannot belong to \(C\), since \(C\) is independent. Thus the endpoints of every positive edge are assigned to distinct vertices. Since \(\dot T(S)\) is complete and either sign is permitted for a positive edge, the required condition holds. Hence $\dot F\to_w\dot T(S)$. Finally, \(C_4\cong K_{2,2}\). When \(a=b=2\), the graph \(S\) is \(2P_3\). Joining an endpoint of one path to an endpoint of the other gives \(P_6\). The corresponding change from \(\dot T(2P_3)\) to \(\dot T(P_6)\) turns one positive edge into a negative edge and leaves all other signs unchanged. The same assignment still satisfies the edge conditions. Therefore, if \(\dot F\to_w\dot T(C_4)\), then \(\dot F\to_w\dot T(P_6)\).
\end{proof}

\section{The complete three-chromatic classification}\label{sec:three}
In this section we prove Theorem~\ref{thm:three} and construct a signed graph for each of its six coefficients. When $\CC(\sF)\ne\varnothing$, the degree bound proved in Theorem~\ref{thm:structure} reduces the admissible graphs to unions of paths and cycles. The final example shows that the five-vertex attainment bound is sharp.

\begin{proof}[Proof of Theorem~\ref{thm:three}]
Part (i) follows from Theorem~\ref{thm:structure}(ii). In the other cases, $k(\sF)\le3$, so every admissible graph has maximum degree at most two. Its connected components are isolated vertices, paths, or cycles. Every cycle has spectral radius two; the degree bound gives $\lambda_1(S)\le2$ for all admissible $S$. This proves (ii).

For (iii), inadmissibility of $C_4$ and Lemma~\ref{lem:twostars} give $\sF\to_w \sT(P_6)$. For \(s=1,2\), every blow-up of \(\dot T(P_s)\) has a bipartite underlying graph and therefore cannot contain a copy of \(\dot F\), since \(\chi(F)=3\). Thus \(P_1\) and \(P_2\) are \(\dot F\)-admissible, so \(t\ge3\). Together with the preceding bound \(t\le6\), this gives $3\le t\le6$. To see that \(C_3\) is not \(\dot F\)-admissible, take a proper three-coloring of \(F\), with color classes \(A_1,A_2,A_3\). The signed graph \(\dot T(C_3)\) is an all-negative triangle. A sufficiently large blow-up of this triangle consists of three independent sets, with all edges between different sets negative. Place the vertices of \(A_i\) at distinct vertices in the \(i\)-th set. Every edge of \(F\) joins different color classes, so every negative edge is represented by a negative edge, and every positive edge is also represented by a negative edge, as permitted by our containment definition. This produces a copy of \(\dot F\) in the blow-up. Hence \(C_3\) is not \(\dot F\)-admissible. The cycles $C_4,C_5$ are excluded by assumption, and every longer cycle contains $P_6$ as an ordinary subgraph, so is also inadmissible.

Consequently every connected component of an admissible graph is a path on at most $t-1$ vertices. Conversely, $P_{t-1}$ is admissible by minimality of $t$. The eigenvalues of $P_s$ are $2\cos(j\pi/(s+1))$, $1\le j\le s$. Thus $L(\sF)=\lambda_1(P_{t-1})=2\cos(\pi/t)$. The attaining graph has at most five vertices in (iii), and $C_4$ or $C_5$ provides such a graph in (ii). The examples below establish occurrence of every $t$ and sharpness.
\end{proof}

\subsection{The rational values}\label{subsec:rational}
The all-negative triangle has no independent negative-edge cover, and therefore has coefficient one. For the all-positive triangle, deleting any vertex leaves a positive edge, so \(k(\dot F)=2\). The bounds in Theorem~\ref{thm:structure}(iii) then give \(\theta(\dot F)=2\).

For coefficient $3/2$, let $\sF_5$ have vertices $b_1,b_2,b_3,a_1,a_2$, positive edges $b_1b_2,\ b_2b_3,\ b_1b_3$ and negative edges $a_1b_1,\ a_1b_2,\ a_2b_2,\ a_2b_3$. There are no other edges. Its underlying graph is three-chromatic, and $\{a_1,a_2\}$ is an independent negative-edge cover. The negative edges form a connected spanning path with bipartition $\{a_1,a_2\}$ and $\{b_1,b_2,b_3\}$. In a negative-edge-preserving map into $C_4$, the three $b_i$ must all lie in the same bipartition class of $C_4$. Their positive triangle requires three distinct images, while that class has only two vertices. Thus $C_4$ is admissible and Theorem~\ref{thm:three} gives $\tht(\sF_5)=3/2$.

\subsection{The value $1+1/\sqrt2$}\label{subsec:F6}
Let $\sF_6$ have vertices $a,b,c,x,y,z$ and edge sets $E^-(\sF_6)=\{ax,by,cz\}$, $E^+(\sF_6)=\{ab,bc,ca,ay,bz,cx\}$. All unlisted pairs are nonedges. The positive triangle gives $\chi(F_6)\ge3$. A proper three-coloring extends the colors of $a,b,c$ by assigning to $x,y,z$ the colors of $b,c,a$, respectively, so equality holds.

Its only independent negative-edge cover is $\{x,y,z\}$. For instance, if such a cover contained $a$, then it would exclude $b$ and hence contain $y$ because $by$ is negative, contradicting the positive edge $ay$. The cases $b,c$ are cyclically identical. By~\eqref{eq:star-test}, we have \(\dot F_6\not\to_w\dot T(P_3)\), since deleting the unique independent negative-edge cover leaves a positive triangle, whose three vertices must be assigned to distinct leaves, whereas \(P_3\) has only two leaves.

With the vertices of $P_4$ labeled consecutively $1,2,3,4$, the map
\[
 (a,b,c,x,y,z)\longmapsto(1,2,3,2,3,4)
\]
realizes $\sF_6\to_w\sT(P_4)$. Since $P_4$ is an ordinary subgraph of both $C_4$ and $C_5$, those cycles are inadmissible. Thus $t=4$, and
\[
 \begin{aligned}
 L(\sF_6)&=\sqrt2,\\
 \tht(\sF_6)&=1+\frac1{\sqrt2}.
 \end{aligned}
\]
The admissible graph $P_3$ attains the spectral supremum. A Perron vector whose entries sum to one is
\[
 \left(1-\frac1{\sqrt2},\ \sqrt2-1,\ 1-\frac1{\sqrt2}\right).
\]
Replace each vertex of $\sT(P_3)$ by an independent set in the corresponding proportion. The resulting signed graphs have weighted density tending to $1$ at $\rho=\tht(\sF_6)$.

\subsection{The golden ratio}\label{subsec:F8}
Let $\sF_8$ have vertices $a,b,c,u,p,q,r,s$ and
\[
 \begin{aligned}
 E^-(\sF_8)&=\{au,bu,ap,bq,cr,cs\},\\
 E^+(\sF_8)&=\{ab,bc,ca,cp,cq,ar,bs\}.
 \end{aligned}
\]
The five vertices outside $\{a,b,c\}$ are independent, and each has neighbors in at most two vertices of that triangle. Hence $\chi(F_8)=3$, and those five vertices form an independent negative-edge cover.

We first show that $\sF_8\not\to_w\sT(P_4)$. Indeed, \(a\) and \(b\) must be assigned distinct vertices of \(P_4\), both adjacent to the vertex assigned to \(u\). By reversing the vertex order of \(P_4\) and, if necessary, interchanging \((a,p,r)\) with \((b,q,s)\), we may assume that \(a\) and \(b\) receive labels \(1\) and \(3\), respectively. Since \(c\) is joined to both \(a\) and \(b\) by positive edges, its label must be \(2\) or \(4\). If \(c\) receives label \(2\), the negative edge \(ap\) forces \(p\) to receive label \(2\) as well, contradicting the requirement that the endpoints of the positive edge \(cp\) have distinct labels. If \(c\) receives label \(4\), the negative edge \(cs\) forces \(s\) to receive label \(3\), the same as \(b\), contradicting the positive edge \(bs\).

The following assignments, in the displayed vertex order, satisfy the required edge conditions:
\[
\begin{array}{c|cccccccc}
 &a&b&c&u&p&q&r&s\\ \hline
 \sT(P_5)&2&4&3&3&1&5&4&2\\
 \sT(C_4)&1&3&2&2&4&4&3&1
\end{array}
\]
Here the cycle labels are cyclic. The first assignment also satisfies the edge conditions in $\sT(C_5)$. Thus $t=5$, and, writing $\varphi=(1+\sqrt5)/2$,
\[
 \begin{aligned}
 L(\sF_8)&=\varphi,\\
 \tht(\sF_8)&=1+\varphi^{-1}=\varphi.
 \end{aligned}
\]
The admissible graph $P_4$ attains the spectral supremum, with normalized Perron vector $(1,\varphi,\varphi,1)/(2+2\varphi)$.

\subsection{The endpoint $t=6$}\label{subsec:F9}
Define $\sF_9$ on $A=\{a_0,a_2,a_3,a_4\}$ and $B=\{b_0,b_1,b_2,b_3,b_4\}$. Its negative edges are
\[
 a_0b_0,\ a_2b_0,\ a_2b_2,\ a_3b_1,\ a_3b_3,\ a_4b_2,\ a_4b_4.
\]
Its positive edges are all pairs $b_i b_j$ with $1\le j-i\le2$, together with
\[
 a_0b_1,\ a_2b_3,\ a_3b_0,\ a_3b_4,\ a_4b_1.
\]
All other pairs are nonedges. In particular, $A$ is an independent negative-edge cover.

Label $P_6$ by $0,1,\ldots,5$. Assign label \(i\) to \(a_i\) and label \(i+1\) to \(b_i\). The endpoints of every negative edge then receive consecutive labels, while those of every positive edge receive labels differing by one or two. Thus $\sF_9\to_w\sT(P_6)$. Reducing these labels modulo any $\ell\ge3$ gives $\sF_9\to_w\sT(C_\ell)$. In particular, \(F_9\) is three-colorable, and the triangle \(b_0b_1b_2\) ensures that \(\chi(F_9)=3\).

We prove that $P_5$ is admissible. Suppose $\sF_9\to_w\sT(P_5)$, and write $x_i$ for the label assigned to $b_i$ in the consecutively labeled path $1,\ldots,5$. The negative common neighbors $a_2,a_3,a_4$, together with the corresponding positive edges within $B$, imply
\begin{equation}\tag{2}\label{eq:F9dist}
 |x_0-x_2|=|x_1-x_3|=|x_2-x_4|=2.
\end{equation}
Thus $x_0,x_2,x_4$ have one parity and $x_1,x_3$ have one parity.

If all five images have the same parity, the triangle $b_0b_1b_2$ requires three distinct labels, so they are $1,3,5$. The overlapping triangles within $B$ force $(x_0,x_1,x_2,x_3,x_4)=(u,v,w,u,v)$. Equation~\eqref{eq:F9dist} would make all three pairwise distances among $u,v,w$ equal to two, which is impossible.

If $x_0,x_2,x_4$ are even and $x_1,x_3$ are odd, then $a_2,a_4$ both have image $3$. The positive edges $a_2b_3$ and $a_4b_1$ exclude $3$ as an image of $b_3,b_1$. Their images are therefore $1$ and $5$, contradicting~\eqref{eq:F9dist}.

In the remaining case $x_1,x_3$ are even. Then $a_3$ has image $3$, and the positive edges $a_3b_0,a_3b_4$ force $x_0,x_4\in\{1,5\}$, so $x_2=3$. Reflect the path to assume $x_0=1$. The negative edge $a_0b_0$ forces label $2$ on $a_0$, and the positive edge $a_0b_1$ gives $x_1=4$, hence $x_3=2$. But $a_2$, a common negative neighbor of images $1$ and $3$, also has image $2$, contradicting the positive edge $a_2b_3$. This proves admissibility of $P_5$.

We have therefore obtained $t=6$, $L(\sF_9)=\sqrt3$, and $\tht(\sF_9)=1+\frac1{\sqrt3}$. The normalized Perron vector of its attaining graph $P_5$ is
\[
 \frac{(1,\sqrt3,2,\sqrt3,1)}{4+2\sqrt3}.
\]
Every admissible component for this example is a path. An admissible graph on at most four vertices then has spectral radius at most $\lambda_1(P_4)=\varphi<\sqrt3$, proving the sharpness of the five-vertex attainment bound.

The examples show why ordinary chromatic number alone cannot determine the signed coefficient. The positive triangle and $\sF_6,\sF_8,\sF_9$ all have underlying chromatic number three, yet give four distinct coefficients. They also illustrate failure of switching invariance: the all-negative triangle has coefficient one, whereas a triangle with exactly one negative edge has an independent cover consisting of an endpoint of that edge and has coefficient two. These two signatures are switching equivalent.

\section{Explicit bounds on admissible components}\label{sec:componentbounds}

Every finite value of $L(\sF)$ is attained, by the mixed-graph theory in the nontrivial case. Indeed, if $\sF$ has a negative edge and $0<L(\sF)<\infty$, Theorem~\ref{thm:spectral} gives $1<\tht(\sF)<\infty$. Let $\mathcal O(\sF)$ comprise all orientations of its negative edges, with its positive edges left undirected. By \cite[Definition~1.3]{MY}, any orientation of the negative edges of a signed host is $\mathcal O(\sF)$-free exactly when that host is $\sF$-free. Thus $\tht(\mathcal O(\sF))=\tht(\sF)$.

The finite-family version of \cite[Theorem~4.13 and Section~5.3]{MY} gives an attaining finite mixed adjacency matrix $(U,D)$ whose balanced blow-ups are $\mathcal O(\sF)$-free. By \cite[Definitions~4.2 and~4.9; proof of Lemma~5.2]{MY}, it has complete underlying graph, $U_{ii}=0$, and $D\ne0$, with directed entries normalized to $2$. Let $A=(D+D^{\mathsf T})/2$ be the adjacency matrix of the ordinary graph $S_*$ recording its directed pairs. Then \(U=J-I-A\). Moreover, \(S_*\) is \(\dot F\)-admissible, since orienting the negative edges of any blow-up of \(\dot T(S_*)\) according to \(D\) gives a mixed graph contained in a sufficiently large balanced blow-up of \((U,D)\). On the simplex,
\[
 \frac{1-x^{\mathsf T}Ux}{x^{\mathsf T}Dx}
 =1+\frac{\|x\|_2^2}{x^{\mathsf T}Ax},
\]
with value $\infty$ when the denominator vanishes. Its minimum is $\tht(\sF)$ by the cited theorem, and equals $1+1/\lambda_1(S_*)$ by the Rayleigh principle and a nonnegative Perron eigenvector. The spectral formula yields $L(\sF)=\lambda_1(S_*)$. If \(\dot F\) is all-positive and \(r=\chi(F)\), then \(S\) is \(\dot F\)-admissible if and only if \(|V(S)|<r\). Hence \(L(\dot F)=r-2\), with equality attained by \(K_{r-1}\). When \(L(\dot F)=0\), the supremum is attained by \(K_1\). When \(L(\dot F)=\infty\), no finite graph attains it.

For integers $D\ge1$ and $s\ge3$, define
\begin{equation}\tag{3}\label{eq:treebound}
 B(D,s)=
 \begin{cases}
  1+D\displaystyle\sum_{j=0}^{(s-4)/2}(D-1)^j,&s\text{ even},\\[4pt]
  2\displaystyle\sum_{j=0}^{(s-3)/2}(D-1)^j,&s\text{ odd}.
 \end{cases}
\end{equation}
Here $0^0=1$, so $B(1,s)=2$ and $B(2,s)=s-1$. For $D\ge2$, $B(D,s)$ is the maximum order of a tree with maximum degree at most $D$ and diameter at most $s-2$~\cite[Theorem~3.1]{CIM}. Unlike general attainment, the following proposition bounds every admissible connected component and gives an explicit finite search.

\begin{proposition}\label{prop:componentbound}
Suppose $\chi(F)\ge3$, $\CC(\sF)\ne\varnothing$, and $\sF\to_w \sT(P_t)$ for some $t$. Let $m=|V(F)|$, $D=k(\sF)-1$, and let $\ell=\min\{t:\sF\to_w \sT(P_t)\}$. Then $3\le\ell\le m$, and every connected component of every $\sF$-admissible graph has at most $B(D,\ell)$ vertices. Consequently, some connected $\sF$-admissible graph $S_*$ satisfies
\[
 \begin{gathered}
 \lambda_1(S_*)=L(\sF),\\
 |V(S_*)|\le B(D,\ell)\le B(D,m).
 \end{gathered}
\]
\end{proposition}

\begin{proof}
The chromatic number assumption gives $\ell\ge3$. Take a map realizing $\sF\to_w\sT(P_t)$. Its used vertices induce disjoint paths and isolated vertices in the ordinary graph $P_t$, on at most $m$ vertices. Joining these path components end to end preserves their edges without merging any vertices. The resulting path has at most \(m\) vertices, so the same vertex assignment satisfies the edge conditions in \(\dot T(P_m)\). Hence $\ell\le m$.

By \eqref{eq:star-test}, every admissible graph has maximum degree at most \(D\). Moreover, no admissible graph contains \(P_\ell\) as a subgraph. For a connected component $Q$, choose a spanning tree $U$. Every path of $U$ is a path of $Q$, so $\operatorname{diam}(U)\le\ell-2$. If \(\ell-2=2a\), there exists a vertex \(v\) such that every vertex is at distance at most \(a\) from \(v\). Counting the vertices at distances \(0,1,\ldots,a\) from \(v\) gives
\[
 |V(Q)|\le 1+D\sum_{j=0}^{a-1}(D-1)^j.
\]
If \(\ell-2=2a+1\) and \(U\) has at least two vertices, deleting a suitable edge leaves two trees, each rooted at an endpoint of the deleted edge and having depth at most \(a\). When \(U\) has even diameter at most \(2a\), any edge incident to a central vertex has this property. Counting the vertices at each distance from the two roots gives
\[
 |V(Q)|\le 2\sum_{j=0}^{a}(D-1)^j.
\]
The bound also holds when \(U\) has only one vertex. These are the two cases of \eqref{eq:treebound}.

The spectral radius of a disjoint union is the maximum of those of its components. Admissibility passes to components, and only finitely many connected graphs have order at most $B(D,\ell)$. Therefore, one of these connected admissible graphs, say \(S_*\), satisfies $\lambda_1(S_*)=L(\dot F)$.
Since \(\ell\le m\) and \(B(D,s)\) is nondecreasing in \(s\) for \(D\ge1\), we also have $|V(S_*)|\le B(D,\ell)\le B(D,m)$. This establishes both the attainment of \(L(\dot F)\) and the stated order bound, completing the proof.
\end{proof}

\begin{corollary}\label{cor:algebraic}
If $0<L(\sF)<\infty$, then $L(\sF)$ is a totally real algebraic integer and $\tht(\sF)$ is a totally real algebraic number. If an attaining graph has order $q$, their common algebraic degree is at most $q$; under Proposition~\ref{prop:componentbound}, it is at most $B(k(\sF)-1,\ell)$. Moreover, if $1<\tht(\sF)<\infty$ is rational and $r=\chi(F)$, then
\[
 \tht(\sF)\in\left\{\frac{r}{r-1},\frac{r-1}{r-2}\right\}.
\]
\end{corollary}

\begin{proof}
The minimal polynomial of $L(\sF)=\lambda_1(S_*)$ divides the characteristic polynomial of the symmetric integer matrix $A(S_*)$. Thus $L(\sF)$ is an algebraic integer with only real conjugates. Its conjugates are nonzero, and
\[
 \begin{gathered}
 \QQ(\tht(\sF))=\QQ(L(\sF)),\\
 \tht(\sF)=1+1/L(\sF),
 \end{gathered}
\]
so the assertions about reality and degree follow. If $\tht(\sF)$ is rational, then $L(\sF)$ is a rational algebraic integer, hence an integer. The bounds $r-2\le L(\sF)\le r-1$ leave precisely the two stated possibilities.
\end{proof}

\section*{Declaration on the use of generative AI}
AI tools were used to assist in checking mathematical results and simplifying proofs. The authors take full responsibility for the content and mathematical correctness of the manuscript.

\end{document}